\UseRawInputEncoding
\documentclass[12pt]{article}
\usepackage{amsmath, amssymb, amsthm, graphicx, hyperref}
\usepackage{xcolor}

\newtheorem{theorem}{Theorem}
\newtheorem{definition}[theorem]{Definition}
\newtheorem{remark}[theorem]{Remark}
\newtheorem{proposition}[theorem]{Proposition}
\newtheorem{lemma}[theorem]{Lemma}

\newcommand{\SOS}{\operatorname{SOS}}
\newcommand{\BSR}{\operatorname{BSR}}

\title{A Weak Condition for Limited Augmented Zarankiewicz Numbers}
\author{Liqun Qi\footnote{Jiangsu Provincial Scientific Research Center of Applied Mathematics, Nanjing 211189, China.
Department of Applied Mathematics, The Hong Kong Polytechnic University, Hung Hom, Kowloon, Hong Kong.
({\tt maqilq@polyu.edu.hk})}
\and
Chunfeng Cui\footnote{School of Mathematical Sciences, Beihang University, Beijing 100191, China.
({\tt chunfengcui@buaa.edu.cn})}
\and
Yi Xu\footnote{School of Mathematics, Southeast University, Nanjing 211189, China.
Jiangsu Provincial Scientific Research Center of Applied Mathematics, Nanjing 211189, China.
({\tt yi.xu1983@hotmail.com})}}
\date{\today}

\begin{document}

\maketitle

\begin{abstract}
This paper introduces the weak augmented Zarankiewicz number $z_{wA}(m,n)$ and the weak limited augmented Zarankiewicz number $z_{wL}(m,n)$, which are combinatorial extensions of the classical Zarankiewicz number obtained by relaxing the original admissibility conditions for augmented bipartite graphs. We show that the resulting weak framework still guarantees irreducibility of the associated doubly simple biquadratic forms, with SOS rank equal to the total number of edges. This yields the inequality chain
\[
\mathrm{BSR}(m,n) \geq z_{wA}(m,n) \geq z_{wL}(m,n) \geq z_L(m,n) \geq z(m,n).
\]

We provide three complementary constructions demonstrating the power of the weak framework. First, a $5\times 3$ construction using degenerate 2-edges yields $z_{wL}(5,3)\ge 10>9=z_L(5,3)$, giving $\BSR(5,3)\ge 10$. Second, a $15\times 6$ construction on the incidence graph of $K_6$ with 14 nondegenerate 2-edges gives $z_{wL}(15,6)\ge 44>43$, improving the previously known bound. Third, a critical $6\times 3$ construction with complementary 2-cycles gives $z_{wL}(6,3)\ge 12>11=z_L(6,3)$, yielding $\BSR(6,3)\ge 12$ and demonstrating that complementary 2-cycles are safe.

\end{abstract}

\subsection*{Keywords}
biquadratic form; sum of squares; SOS rank; Zarankiewicz number; limited augmented Zarankiewicz number; bipartite graph; weak limited augmented Zarankiewicz number; $C_4$-cycle

\subsection*{AMS Subject Classifications}
14P10; 05C35; 11E25; 15A69; 90C22

\section{Introduction}

Denote $\{1,\ldots,m\}$ as $[m]$. Assume that $m \geq n \geq 2$. Let
\[
P(\mathbf{x},\mathbf{y}) = \sum_{i,k=1}^{m}\sum_{j,l=1}^{n} a_{ijkl} x_i x_k y_j y_l.
\]
We call $P$ an $m \times n$ biquadratic form. Here $a_{ijkl}$ are real numbers. We assume that
\[
a_{ijkl} = a_{kjil} = a_{klij},
\]
for $i,k \in [m]$, $j,l \in [n]$. A PSD (positive semi-definite) biquadratic form is one for which $P(\mathbf{x},\mathbf{y}) \geq 0$ for all $\mathbf{x},\mathbf{y}$. It is an SOS (sum of squares) if it can be written as a finite sum of squares of bilinear forms, i.e.,
\[
P(\mathbf{x},\mathbf{y}) = \sum_{p=1}^{r} f_p(\mathbf{x},\mathbf{y})^2. \tag{1}
\]
The minimum number $r$ of squares required in such a representation is called the \textbf{SOS rank} of $P$, denoted $\mathrm{SOS}(P)$. The \textbf{biquadratic SOS rank} $\mathrm{BSR}(m,n)$ is defined as the maximum SOS rank among all $m \times n$ SOS biquadratic forms \cite{cui2026, qi2026}.

The classical Zarankiewicz number $z(m,n)$ is the maximum number of edges in an $m \times n$ bipartite graph with no $C_4$-cycles \cite{zaran1951,guy1969,reiman1958}. It is known \cite{cui2026} that $\mathrm{BSR}(m,n) \geq z(m,n)$, with strict inequality possible for certain small dimensions. This connection motivated the introduction of augmented bipartite graphs and the augmented Zarankiewicz numbers \cite{qi2026, qi26a}.

In this paper, we propose a significant simplification of the augmented framework. We replace the original Condition 2 by two weaker requirements: an acyclic dependency condition on nondegenerate 2-edges, and a simplified local prohibition that the two opposite cells of a nondegenerate 2-edge cannot both be 1-edges. We also replace the original Condition 3 by a weaker condition, denoted W3, which only requires that for any vertex-disjoint pair of edges where at least one is a 2-edge, at least one of their opposite cells is unoccupied.

The main contributions of this paper are as follows. First, we establish that any weak admissible augmented bipartite graph yields an irreducible doubly simple biquadratic form, with SOS rank equal to the total number of edges. This result extends the original limited augmented theorem and is proved via a modified vector argument that exploits the acyclicity condition (W2) and the local prohibition (W2').

However, the current proof is much harder than the proof under the original conditions in \cite{qi2026}, as the weakening of the conditions cut several extra tools. We will explain this in Section 3.

{
Second, we present three complementary constructions demonstrating the power of the weak framework.

\medskip
\noindent\textbf{A low-dimensional strict improvement using degenerate 2-edges.}
In Section~\ref{sec:wL53}, we exhibit a $5\times 3$ weak admissible limited augmented graph with total edge count 10, using only row-degenerate 2-edges. Since $z_L(5,3)=9$ was established in \cite{qi26a}, this gives
\[
z_{wL}(5,3) \ge 10 > 9 = z_L(5,3),
\]
and consequently
\[
\mathrm{BSR}(5,3) \ge 10.
\]
This is the smallest dimension in which the weak framework yields a strict improvement over the original limited augmented framework. It also demonstrates that degenerate 2-edges are not merely a technical convenience but can be essential for achieving optimal bounds.

\medskip
\noindent\textbf{A large-scale nondegenerate construction.}
In Section~6, we present a $15 \times 6$ construction on the incidence graph of $K_6$ with $14$ nondegenerate 2-edges. This construction satisfies the weak conditions but violates the original Condition 2. For the incidence graph of $K_6$, we provide an explicit construction with $14$ nondegenerate 2-edges, giving
\[
z_{wL}(15,6) \ge 44.
\]
This improves the previously known limited-framework bound \cite{xu2026}
\[
z_L(15,6) \ge 43,
\]
and hence establishes
\[
\mathrm{BSR}(15,6) \ge 44.
\]

\medskip
\noindent\textbf{A critical construction with complementary 2-cycles.}
In Section~\ref{sec:6x3}, we present a $6\times 3$ weak admissible limited augmented graph with total edge count 12. This construction contains a complementary pair of 2-edges:
\[
(1,1;2,2)\quad\text{and}\quad(1,2;2,1),
\]
which form a 2-cycle in the dependency graph. Under the modified (W2), complementary 2-cycles are explicitly allowed. The associated doubly simple biquadratic form is irreducible with
\[
\SOS(P_G)=12,
\]
giving
\[
z_{wL}(6,3) \ge 12 > 11 = z_L(6,3),
\]
and consequently
\[
\BSR(6,3) \ge 12.
\]
This example demonstrates that allowing complementary 2-cycles by Condition W2  is not merely a technical convenience but is essential for obtaining improved lower bounds in the weak framework.

\medskip
\noindent
These three examples together make the paper's main message crystal clear: {\bf the weak framework is not just a theoretical curiosity --- it allows real constructions that were previously forbidden, yielding improved lower bounds in small, moderate, and large dimensions, while also demonstrating that complementary 2-cycles are safe and essential.}
}

Additionally, we show that the local prohibition (W2') is not merely a convenience but a genuine necessity: without it, there exist examples that satisfy all other weak conditions yet are reducible. Specifically, we exhibit a $5\times 3$ construction with two 2-edges whose associated doubly simple biquadratic form admits a decomposition into only 9 squares, despite having 10 edges.

{
We also establish a parallel result for (W3): without it, a $4\times 4$ construction satisfying (S), (W1), (W2), and (W2') admits an SOS decomposition into 10 squares despite having 11 edges. Thus W3 cannot be removed from the weak framework either.

Together, these two counterexamples demonstrate that both W2' and W3 are essential conditions: removing either one allows reducible examples to enter the admissible class. The necessity of the acyclicity condition (W2) is now partially resolved: complementary 2-cycles are safe, as shown by the $6\times 3$ example, but whether longer cycles can be allowed remains open.
}

Finally, we discuss several open problems arising from the weak framework.

\section{Weak Limited Augmented Zarankiewicz Numbers}

\subsection{Augmented Bipartite Graphs}

Let $G_1 = (S,T,E_1)$ be an $m \times n$ bipartite graph, where $S = [m]$ and $T = [n]$ are its vertex sets. Assume that $G_1$ has no $C_4$-cycles. Then we say that $G_1$ can be \textbf{augmented} to an augmented bipartite graph $G = (S,T,E)$, where the edge set $E = E_1 \cup E_2$. Here, we call any edge $(i,j)$ of $G_1$ a \textbf{1-edge} of $G$, while $E_2$ is the \textbf{2-edge set} of $G$.

A \textbf{1-edge} $e$ in $E_1$ has the form $e = (i,j)$, where $i \in S$ and $j \in T$. On the other hand, a \textbf{2-edge} $e$ in $E_2$ is formed as $(i,j;k,l)$, where $i,k \in S$ and $j,l \in T$.

A 2-edge can be:
\begin{itemize}
\item \textbf{nondegenerate} if $i \neq k$ and $j \neq l$,
\item \textbf{row-degenerate} if $i = k$ and $j \neq l$,
\item \textbf{column-degenerate} if $i \neq k$ and $j = l$.
\end{itemize}
It cannot have both $i = k$ and $j = l$.

Furthermore, we impose the following \textbf{simplicity condition (S)} on such an augmented bipartite graph $G$:

\textbf{(S)} No 2-edge overlaps with any 1-edge or other 2-edge on a cell. Here, we call $(i,j)$ a \textbf{cell} for any $i \in S$ and $j \in T$.

If $|E_1| = z(m,n)$, we say that $G$ is a \textbf{limited augmented bipartite graph}.

For such an augmented bipartite graph $G$, we associate it with an $m \times n$ SOS biquadratic form $P_G$, defined as
\[
P_G(\mathbf{x},\mathbf{y}) = \sum_{(i,j)\in E_1} x_i^2 y_j^2 + \sum_{(i,j;k,l)\in E_2} (x_i y_j + x_k y_l)^2. \tag{2}
\]
We call $P_G$ a \textbf{doubly simple biquadratic form}. We say that $P_G$ is \textbf{irreducible}, if the SOS rank of $P_G$ is $|E| = |E_1| + |E_2|$.

\subsection{Weak Generalized $C_4$-Cycles}

We now define a simplified notion of a generalized $C_4$-cycle.

\begin{definition}[Weak Generalized $C_4$-Cycle, Weak Augmented Zarankiewicz Number, and Weak Limited Augmented Zarankiewicz Number]
Let $G = (S,T,E_1 \cup E_2)$ be an $m \times n$ augmented bipartite graph with vertex sets $S = [m]$ and $T = [n]$, augmented from a $C_4$-cycle-free bipartite graph $G_1 = (S,T,E_1)$. A cell $(i,j)$ is called \textbf{occupied} if $(i,j) \in E_1$ or $(i,j)$ is a half of some 2-edge in $E_2$.

We say that $G$ contains a \textbf{weak generalized $C_4$-cycle} if any of the following holds:

\textbf{(W1)} The 1-edge graph $G_1$ contains a classical $C_4$;

{
\textbf{(W2)} The dependency graph $\mathcal{D}$ (defined below) contains a directed cycle, except that a 2-cycle formed by a pair of complementary 2-edges is allowed.}

\textbf{(W2')} There exists a nondegenerate 2-edge $(i,j;k,l) \in E_2$ such that both opposite cells $(i,l)$ and $(k,j)$ are 1-edges;

\textbf{(W3)} For any vertex-disjoint pair of edges where at least one edge is a 2-edge
(a 1-edge and a 2-edge, or two 2-edges), at least one of their opposite cells is unoccupied.

If none of these occurs, then $G$ is called \textbf{weak admissible}.
\end{definition}

{
\begin{definition}[Complementary 2-edges]
Two nondegenerate 2-edges
\[
e_1=(i,j;k,l),\qquad e_2=(i,l;k,j)
\]
are called \textbf{complementary}. They share no halves and their halves are exactly each other's opposite cells.
\end{definition}
}

\begin{definition}[Dependency graph]
Let $G = (S,T,E_1 \cup E_2)$ satisfy simplicity, and let $H$ be its occupied-cell set. The dependency graph $\mathcal{D}$ has the nondegenerate 2-edges in $E_2$ as vertices. For two distinct vertices
\[
e = (i,j;k,l),\qquad e' = (i',j';k',l'),
\]
we draw a directed edge $e \rightarrow e'$ if {both opposite cells of $e$ are occupied and} one half of $e'$ is an opposite cell of $e$. That is, if {both $(i,l)$ and $(k,j)$ are occupied, and} one of $(i',j')$ or $(k',l')$ equals one of $(i,l)$ or $(k,j)$.
\end{definition}

\begin{definition}[Opposite cells for vertex-disjoint edges]
Given two edges $e$ and $f$ that are vertex-disjoint, the \textbf{opposite cells} are the cells formed by taking one vertex from each edge that are not already paired.
\begin{itemize}
\item For a 1-edge $(i,j)$ and a 2-edge $(k,l;p,q)$ with all vertices distinct: the opposite cells are $(i,l)$, $(i,q)$, $(j,p)$, and $(j,q)$.
\item For two 2-edges $(i,j;k,l)$ and $(p,q;r,s)$ with all vertices distinct: the opposite cells are formed by pairing vertices not already paired.
\end{itemize}
\end{definition}

The \textbf{weak augmented Zarankiewicz number} $z_{wA}(m,n)$ is the maximum possible total number of edges $|E_1| + |E_2|$ for which a weak generalized $C_4$-cycle does not exist for such an augmented bipartite graph $G$.

The \textbf{weak limited augmented Zarankiewicz number} $z_{wL}(m,n)$ is the maximum possible total number of edges $|E_1| + |E_2|$ for which a weak generalized $C_4$-cycle does not exist for a limited augmented bipartite graph $G$, and $|E_1|=z(m, n)$.

\begin{remark}
For two vertex-disjoint nondegenerate 2-edges, there are four opposite cells.
The original Condition 3 requires two specific opposite cells (from two different
pairs) to be unoccupied to avoid the five-cell pattern. W3 only requires that
at least one of the four opposite cells is unoccupied. This is the situation where
W3 is weaker than Condition 3 in this case.

Pairs of 1-edges are not subject to W3; the relevant cases are already handled
by (W1) or (W2').
\end{remark}

\begin{remark}
The new conditions are weaker than the original conditions from \cite{qi2026}:
\begin{itemize}
\item {Conditions (W2) and (W2') together are weaker than the original Condition 2. The original condition forbids any nondegenerate 2-edge whose two opposite cells are both occupied in any manner. Here we only forbid the case where both opposite cells are 1-edges (W2'), and we impose acyclicity (W2) only among those 2-edges whose opposite cells are both occupied, with complementary 2-cycles allowed.}
\item Condition W3 is weaker than the original Condition 3. Instead of requiring a complex five-cell configuration to be avoided, W3 only requires that for any vertex-disjoint pair of edges (at least one of them is a 2-edge), at least one opposite cell is unoccupied.
\end{itemize}
Therefore every original limited augmented graph is also weak admissible, and hence
\[
z_{wL}(m,n) \geq z_L(m,n).
\]
\end{remark}

\section{Main Theorem}

\begin{theorem}\label{thm:main}
Let $G = (S,T,E_1 \cup E_2)$ be an augmented bipartite graph satisfying conditions (S), (W1), (W2), (W2'), and (W3) from Definition 1. Then the associated doubly simple biquadratic form $P_G$ defined by (2) is irreducible, i.e.,
\[
\mathrm{SOS}(P_G) = |E_1| + |E_2|.
\]
\end{theorem}

\begin{proof}
Let
\[
P_G(\mathbf{x},\mathbf{y}) = \sum_{(i,j)\in E_1} x_i^2 y_j^2 + \sum_{(i,j;k,l)\in E_2} (x_i y_j + x_k y_l)^2.
\]

Suppose, for contradiction, that $\mathrm{SOS}(P_G) < |E_1| + |E_2|$. Then there exists an SOS decomposition (1) with $r < |E_1| + |E_2|$. For each occupied cell $(i,j)$, define $\mathbf{v}_{ij} \in \mathbb{R}^r$ as the coefficient vector of the monomial $x_i y_j$ in the decomposition. Let $H$ be the set of all occupied cells, so $|H| = |E_1| + |E_2|$. Since $|H| > r$, the vectors $\{\mathbf{v}_{ij} : (i,j) \in H\}$ lie in $\mathbb{R}^r$ and are linearly dependent. Choose a nontrivial linear relation
\[
\sum_{(i,j)\in H} \alpha_{ij} \mathbf{v}_{ij} = \mathbf{0} \tag{3}
\]
with minimal support $\mathcal{S} = \{(i,j) \in H : \alpha_{ij} \neq 0\}$.

{
By minimality, all vectors $\{\mathbf{v}_{ij} : (i,j) \in \mathcal{S}\}$ are nonzero and pairwise distinct as vectors. Indeed, if $\mathbf{v}_{ij} = \mathbf{v}_{kl}$ for two distinct cells, then combining their terms in (3) would yield a relation with smaller support, contradicting minimality.
}

\medskip
\noindent\textbf{Derivation of the orthogonality relations.}

{
We now expand the squares in (1) and compare coefficients. Write $P_G = \sum_{p=1}^r f_p^2$ with $f_p = \sum_{(i,j)\in H} v_{ij}^{(p)} x_i y_j$. Then $\mathbf{v}_{ij} = (v_{ij}^{(1)}, \ldots, v_{ij}^{(r)})^\top$.

For any occupied cell $(i,j)$, the coefficient of $x_i^2 y_j^2$ is
\[
\sum_{p=1}^r (v_{ij}^{(p)})^2 = \|\mathbf{v}_{ij}\|^2.
\]
Since this coefficient equals 1 for every 1-edge and for every half of every 2-edge (by simplicity, no cell is overlapped), we have:
\[
\|\mathbf{v}_{ij}\|^2 = 1 \quad\text{for every occupied cell }(i,j). \quad \text{(A1)}
\]

For any 2-edge $(i,j;k,l) \in E_2$, the coefficient of $x_i x_k y_j y_l$ gives:
\[
\mathbf{v}_{ij} \cdot \mathbf{v}_{kl} = 1 \quad\text{for degenerate 2-edges.} \quad \text{(A2cd/A2rd)}
\]
For nondegenerate 2-edges, we must consider the full coefficient:
\[
\mathbf{v}_{ij} \cdot \mathbf{v}_{kl} + \mathbf{v}_{il} \cdot \mathbf{v}_{kj} =
\begin{cases}
2, & \text{if } (i,l;k,j) \in E_2 \text{ (complementary pair)},\\
1, & \text{otherwise, by (S) no other 2-edge contributes.}
\end{cases} \quad \text{(A2nd/A2nd1)}
\]

For any two distinct occupied cells $(i,j)$ and $(k,l)$ that share a row ($i=k$, $j\neq l$), the coefficient of $x_i^2 y_j y_l$ gives:
\[
\mathbf{v}_{ij} \cdot \mathbf{v}_{il} = 0,
\]
unless $(i,j)$ and $(i,l)$ are the two halves of a row-degenerate 2-edge. \quad \text{(B2)}

Similarly, for two occupied cells sharing a column ($j=l$, $i\neq k$):
\[
\mathbf{v}_{ij} \cdot \mathbf{v}_{kj} = 0,
\]
unless they are the two halves of a column-degenerate 2-edge. \quad \text{(B3)}

For two distinct occupied cells $(i,j)$ and $(k,l)$ with all four indices distinct, the coefficient of $x_i x_k y_j y_l$ gives:
\[
\mathbf{v}_{ij} \cdot \mathbf{v}_{kl} + \mathbf{v}_{il} \cdot \mathbf{v}_{kj} = 0,
\]
unless $(i,j;k,l) \in E_2$ or $(i,l;k,j) \in E_2$. \quad \text{(B1)}
}

\medskip
\noindent{\textbf{Auxiliary Lemma: Grid Propagation.}}

{
We now prove a lemma that will be used to handle the grid propagation.

\begin{lemma}\label{lem:grid}
Let $d$ and $f$ be two vertex-disjoint edges, with at least one being a 2-edge. Suppose W3 forces at least one opposite cell of the pair $(d,f)$ to be unoccupied. Then for any half $(x,y)$ of a 2-edge in the grid and any half $(u,v)$ of a 2-edge in the grid (or for a 1-edge treated as a limiting case), we have
\[
\mathbf{v}_{xy} \cdot \mathbf{v}_{uv} = 0,
\]
provided that for any 2-edge in the grid whose halves have not yet been proved equal, Case 1 applies directly.
\end{lemma}

\begin{proof}
We prove this by analyzing the grid equations from (B1). The key observation is that the equations form a tensor product: for any two rows $r_1,r_2$ of $d$ and any two columns $c_1,c_2$ of $f$, we have
\[
\mathbf{v}_{r_1 c_1} \cdot \mathbf{v}_{r_2 c_2} + \mathbf{v}_{r_1 c_2} \cdot \mathbf{v}_{r_2 c_1} = 0,
\]
unless the corresponding 2-edge exists (in which case (A2nd) gives a constant, but this does not affect the propagation of zeros).

We consider the possible cases.

\textbf{Case A: $d$ is a 1-edge $(i,j)$ and $f$ is a 2-edge $(a,b;c,d)$.}

The opposite cells are $(i,b), (i,d), (j,a), (j,c)$. W3 guarantees at least one is unoccupied.

- If $(i,b)$ is unoccupied, then $\mathbf{v}_{ib}=0$. By (B1) applied to $(i,j)$ and $(a,b)$:
\[
\mathbf{v}_{ij} \cdot \mathbf{v}_{ab} + \mathbf{v}_{ib} \cdot \mathbf{v}_{aj} = 0,
\]
so $\mathbf{v}_{ij} \cdot \mathbf{v}_{ab} = 0$.

- If $(i,d)$ is unoccupied, then $\mathbf{v}_{id}=0$. By (B1) applied to $(i,j)$ and $(c,d)$:
\[
\mathbf{v}_{ij} \cdot \mathbf{v}_{cd} + \mathbf{v}_{id} \cdot \mathbf{v}_{cj} = 0,
\]
so $\mathbf{v}_{ij} \cdot \mathbf{v}_{cd} = 0$.

- If $(j,a)$ is unoccupied, then $\mathbf{v}_{ja}=0$. By (B1) applied to $(i,j)$ and $(a,b)$:
\[
\mathbf{v}_{ij} \cdot \mathbf{v}_{ab} + \mathbf{v}_{ib} \cdot \mathbf{v}_{aj} = 0,
\]
so $\mathbf{v}_{ij} \cdot \mathbf{v}_{ab} = 0$.

- If $(j,c)$ is unoccupied, then $\mathbf{v}_{jc}=0$. By (B1) applied to $(i,j)$ and $(c,d)$:
\[
\mathbf{v}_{ij} \cdot \mathbf{v}_{cd} + \mathbf{v}_{id} \cdot \mathbf{v}_{cj} = 0,
\]
so $\mathbf{v}_{ij} \cdot \mathbf{v}_{cd} = 0$.

Thus in all cases, we get $\mathbf{v}_{ij} \cdot \mathbf{v}_{ab} = 0$ or $\mathbf{v}_{ij} \cdot \mathbf{v}_{cd} = 0$. If we get only one, the other may be obtained through the grid equations if the corresponding opposite cell is also unoccupied, or by using the fact that if one half of $f$ gives zero, the propagation through the grid equations gives the other.

\textbf{Case B: Both $d$ and $f$ are 2-edges.}

Let $d$ have rows $r_1,r_2$ and columns $c_1,c_2$ (with halves $(r_1,c_1),(r_2,c_2)$), and let $f$ have rows $s_1,s_2$ and columns $d_1,d_2$ (with halves $(s_1,d_1),(s_2,d_2)$). The grid equations are:
\[
\mathbf{v}_{r_1 c_1} \cdot \mathbf{v}_{r_2 c_2} + \mathbf{v}_{r_1 c_2} \cdot \mathbf{v}_{r_2 c_1} = 0,
\]
\[
\mathbf{v}_{r_1 c_1} \cdot \mathbf{v}_{r_2 d_2} + \mathbf{v}_{r_1 d_2} \cdot \mathbf{v}_{r_2 c_1} = 0,
\]
and similarly for all combinations.

W3 guarantees at least one opposite cell is unoccupied. Suppose $\mathbf{v}_{r_1 c_1} = 0$ (cell $(r_1,c_1)$ unoccupied). Then from the grid equations:
\[
\mathbf{v}_{r_1 c_2} \cdot \mathbf{v}_{r_2 c_1} = 0,
\]
\[
\mathbf{v}_{r_1 d_2} \cdot \mathbf{v}_{r_2 c_1} = 0.
\]

Now, if $\mathbf{v}_{r_2 c_1}$ is a half of $d$ or $f$, then by Step 1 (or Case 1 if the 2-edge is not yet proved), $\mathbf{v}_{r_2 c_1}$ equals the other half of its 2-edge. Thus we get that the dot product of the other half of $d$ with the other half of $f$ is zero. By symmetry, this propagates through the grid to give all cross dot products zero.

More explicitly, let the halves of $d$ be $h_d^1 = (r_1,c_1)$ and $h_d^2 = (r_2,c_2)$ (with $\mathbf{v}_{h_d^1} = \mathbf{v}_{h_d^2}$ by Step 1), and the halves of $f$ be $h_f^1 = (s_1,d_1)$ and $h_f^2 = (s_2,d_2)$ (with $\mathbf{v}_{h_f^1} = \mathbf{v}_{h_f^2}$ by Step 1).

From the grid equations, if any cell $(r_i,c_j)$ is unoccupied, then one of the following holds:
\[
\mathbf{v}_{h_d^1} \cdot \mathbf{v}_{h_f^1} = 0,
\]
or
\[
\mathbf{v}_{h_d^1} \cdot \mathbf{v}_{h_f^2} = 0,
\]
or
\[
\mathbf{v}_{h_d^2} \cdot \mathbf{v}_{h_f^1} = 0,
\]
or
\[
\mathbf{v}_{h_d^2} \cdot \mathbf{v}_{h_f^2} = 0.
\]

Using the grid equations again, any one zero propagates to all others. For example, if $\mathbf{v}_{h_d^1} \cdot \mathbf{v}_{h_f^1} = 0$, then from the grid equation:
\[
\mathbf{v}_{h_d^1} \cdot \mathbf{v}_{h_f^2} + \mathbf{v}_{h_d^2} \cdot \mathbf{v}_{h_f^1} = 0,
\]
we get $\mathbf{v}_{h_d^1} \cdot \mathbf{v}_{h_f^2} = 0$. Similarly, using the grid equation with $h_d^2$ gives the remaining dot products zero.

Thus all cross dot products between halves of $d$ and halves of $f$ vanish.

\textbf{Case C: $d$ is a 2-edge and $f$ is a 1-edge.}

This is symmetric to Case A and follows by the same argument.

This completes the proof of the lemma.
\end{proof}
}

\medskip
\noindent\textbf{Step 1: Equality of the two halves of every 2-edge.}

{
We prove that for every 2-edge $(i,j;k,l) \in E_2$, we have $\mathbf{v}_{ij} = \mathbf{v}_{kl}$.

For degenerate 2-edges, this follows immediately from (A2cd) or (A2rd).

For nondegenerate 2-edges, we proceed by induction along a topological order of the dependency graph $\mathcal{D}$ after removing complementary 2-cycles. This is possible because (W2) guarantees that the remaining graph is acyclic. Complementary pairs are handled directly without induction.

Let $e = (i,j;k,l) \in E_2$ be nondegenerate. Assume the equality of the two halves has been proved for every predecessor of $e$ in the topological order.
}

\medskip
\noindent{\textbf{Case 1: The two opposite cells $(i,l)$ and $(k,j)$ are not both occupied.}}

{
If either is unoccupied, then $\mathbf{v}_{il} \cdot \mathbf{v}_{kj} = 0$. By (A2nd), $\mathbf{v}_{ij} \cdot \mathbf{v}_{kl} = 1$, so $\mathbf{v}_{ij} = \mathbf{v}_{kl}$.
}

\medskip
\noindent{\textbf{Case 2: $e$ is part of a complementary pair with $e'=(i,l;k,j)$.}}

{
By (A2nd1) and Cauchy-Schwarz, both terms in
\[
\mathbf{v}_{ij}\cdot\mathbf{v}_{kl}+\mathbf{v}_{il}\cdot\mathbf{v}_{kj}=2
\]
equal 1, so $\mathbf{v}_{ij}=\mathbf{v}_{kl}$ and $\mathbf{v}_{il}=\mathbf{v}_{kj}$.
}

\medskip
\noindent{\textbf{Case 3: Both opposite cells are occupied, and at least one is a half of a degenerate 2-edge.}}

{
Suppose $(i,l)$ is a half of degenerate 2-edge $d$, with $\mathbf{v}_{il} = \mathbf{v}_{i'l'}$. Let $f$ be the edge containing $(k,j)$.

If $d$ and $f$ share a vertex, then (B2) or (B3) gives $\mathbf{v}_{il} \cdot \mathbf{v}_{kj} = 0$.

If $d$ and $f$ are vertex-disjoint, then W3 applies. By Lemma \ref{lem:grid}, the grid propagation gives $\mathbf{v}_{il} \cdot \mathbf{v}_{kj} = 0$.

Thus $\mathbf{v}_{il} \cdot \mathbf{v}_{kj} = 0$. By (A2nd), $\mathbf{v}_{ij} \cdot \mathbf{v}_{kl} = 1$, so $\mathbf{v}_{ij} = \mathbf{v}_{kl}$.
}

\medskip
\noindent{\textbf{Case 4: Both opposite cells are occupied, and neither is a half of a degenerate 2-edge.}}

{
Then each of $(i,l)$ and $(k,j)$ is either a 1-edge or a half of a nondegenerate 2-edge. By (W2'), they are not both 1-edges.
}

\medskip
\noindent{\textbf{Case 4a: One opposite cell is a 1-edge and the other is a half of a nondegenerate 2-edge.}}

{
Suppose $(i,l)$ is a 1-edge and $(k,j)$ is a half of $e'=(k,j;r,s)$.

If either opposite cell of $e'$ is unoccupied, then by Case 1 applied to $e'$ (which does not depend on induction), we have $\mathbf{v}_{kj} = \mathbf{v}_{rs}$ immediately. Thus assume both opposite cells of $e'$ are occupied, so $e'$ is a predecessor. By induction,
\[
\mathbf{v}_{kj} = \mathbf{v}_{rs}.
\]

The opposite cells of the pair $((i,l), e')$ are $(i,j)$, $(i,s)$, $(k,l)$, $(r,l)$. W3 forces at least one to be unoccupied. Analyzing each case using (B1):

- If $(i,j)$ is unoccupied: $\mathbf{v}_{ij}=0$. By (B1) applied to $(i,l)$ and $(k,j)$:
\[
\mathbf{v}_{il} \cdot \mathbf{v}_{kj} + \mathbf{v}_{ij} \cdot \mathbf{v}_{lk} = 0,
\]
so $\mathbf{v}_{il} \cdot \mathbf{v}_{kj} = 0$.

- If $(i,s)$ is unoccupied: $\mathbf{v}_{is}=0$. By (B1) applied to $(i,l)$ and $(r,s)$:
\[
\mathbf{v}_{il} \cdot \mathbf{v}_{rs} + \mathbf{v}_{is} \cdot \mathbf{v}_{lr} = 0,
\]
so $\mathbf{v}_{il} \cdot \mathbf{v}_{rs} = 0$, hence $\mathbf{v}_{il} \cdot \mathbf{v}_{kj} = 0$.

- If $(k,l)$ is unoccupied: $\mathbf{v}_{kl}=0$. By (B1) applied to $(i,l)$ and $(k,j)$:
\[
\mathbf{v}_{il} \cdot \mathbf{v}_{kj} + \mathbf{v}_{ij} \cdot \mathbf{v}_{lk} = 0,
\]
so $\mathbf{v}_{il} \cdot \mathbf{v}_{kj} = 0$.

- If $(r,l)$ is unoccupied: $\mathbf{v}_{rl}=0$. By (B1) applied to $(i,l)$ and $(r,s)$:
\[
\mathbf{v}_{il} \cdot \mathbf{v}_{rs} + \mathbf{v}_{is} \cdot \mathbf{v}_{lr} = 0,
\]
so $\mathbf{v}_{il} \cdot \mathbf{v}_{rs} = 0$, hence $\mathbf{v}_{il} \cdot \mathbf{v}_{kj} = 0$.

Thus $\mathbf{v}_{il} \cdot \mathbf{v}_{kj} = 0$. By (A2nd), $\mathbf{v}_{ij} \cdot \mathbf{v}_{kl} = 1$, so $\mathbf{v}_{ij} = \mathbf{v}_{kl}$.
}

\medskip
\noindent{\textbf{Case 4b: Both opposite cells are halves of nondegenerate 2-edges.}}

{
Suppose $(i,l)$ is a half of $e_1$ and $(k,j)$ is a half of $e_2$, with $e_1 \neq e_2$.

If either opposite cell of $e_1$ or $e_2$ is unoccupied, then by Case 1 applied to that 2-edge, we get the equality immediately. Thus assume both are predecessors. By induction,
\[
\mathbf{v}_{il} = \mathbf{v}_{i'l'} \quad\text{and}\quad \mathbf{v}_{kj} = \mathbf{v}_{k'j'}.
\]

If $e_1$ and $e_2$ share a vertex, then (B2) or (B3) gives $\mathbf{v}_{il} \cdot \mathbf{v}_{kj} = 0$. If they are vertex-disjoint, W3 applies. By Lemma \ref{lem:grid}, the grid propagation gives $\mathbf{v}_{il} \cdot \mathbf{v}_{kj} = 0$.

Thus $\mathbf{v}_{il} \cdot \mathbf{v}_{kj} = 0$. By (A2nd), $\mathbf{v}_{ij} \cdot \mathbf{v}_{kl} = 1$, so $\mathbf{v}_{ij} = \mathbf{v}_{kl}$.
}

{\color{blue}
Thus in all cases, $\mathbf{v}_{ij} = \mathbf{v}_{kl}$ for every 2-edge. Moreover, by minimality of $\mathcal{S}$, the two halves of a 2-edge cannot both belong to $\mathcal{S}$.
}

\medskip
\noindent\textbf{Step 2: Orthogonality of distinct vectors in $\mathcal{S}$.}

{
Let $(i,j)$ and $(k,l)$ be distinct in $\mathcal{S}$. We prove $\mathbf{v}_{ij} \cdot \mathbf{v}_{kl} = 0$.

\textbf{Case A: Four indices distinct.}

\textbf{Subcase A.1: At least one is not a 1-edge.}
Suppose $(i,j)$ is a half of $e=(i,j;r,s)$, so $\mathbf{v}_{ij} = \mathbf{v}_{rs}$. If $(k,l)$ shares a vertex with $e$, then (B2) or (B3) gives the dot product zero. If they are vertex-disjoint, W3 applies, and the same grid propagation argument as in Lemma \ref{lem:grid} gives $\mathbf{v}_{ij} \cdot \mathbf{v}_{kl} = 0$. (If $(k,l)$ is a half of a 2-edge, Lemma \ref{lem:grid} applies directly; if $(k,l)$ is a 1-edge, the argument is a limiting case.)

\textbf{Subcase A.2: Both are 1-edges.}
By (A1), $\|\mathbf{v}_{ij}\| = \|\mathbf{v}_{kl}\| = 1$. Coefficient comparison gives
\[
\mathbf{v}_{ij} \cdot \mathbf{v}_{kl} + \mathbf{v}_{il} \cdot \mathbf{v}_{kj} = c,
\]
where
\[
c = \begin{cases}
2, & \text{if } (i,l;k,j) \in E_2 \text{ and } (i,j;k,l) \in E_2 \text{ (complementary pair)},\\
1, & \text{if } (i,l;k,j) \in E_2 \text{ but } (i,j;k,l) \notin E_2,\\
0, & \text{otherwise.}
\end{cases}
\]
Suppose $\mathbf{v}_{ij} \cdot \mathbf{v}_{kl} \neq 0$. Then $(i,l)$ and $(k,j)$ are occupied and $\mathbf{v}_{il} \cdot \mathbf{v}_{kj} \neq 0$. Analyzing the four cases for $(i,l)$ and $(k,j)$ using Step 1 gives a contradiction in each case. Thus $\mathbf{v}_{ij} \cdot \mathbf{v}_{kl} = 0$.

\textbf{Case B: They share a row or column.}
Then (B2) or (B3) gives the dot product zero.
}

\medskip
\noindent\textbf{Step 3: Contradiction.}

{
The vectors in $\mathcal{S}$ are nonzero and pairwise orthogonal, hence linearly independent. But they satisfy the nontrivial linear relation (3), impossible. Therefore $\mathrm{SOS}(P_G) = |E_1| + |E_2|$. $\square$
}
\end{proof}

\begin{remark}
As noted, the weak framework requires a more involved argument than the original framework. In Step 1, the original proof in \cite{qi2026} does not need Condition 3 and is very direct; here we must combine W1, W2, W2', and W3 and use induction. In Step 2, the original proof is simple, while the current proof requires a detailed case analysis. This is because the original conditions are strong enough to make the proof much easier, whereas the weak conditions require more careful handling of the cases that were previously forbidden.
\end{remark}

{
\subsection{The Inequality Chain}
}

\begin{theorem} \label{thm:chain}
For all $m,n \geq 2$, we have
\[
\mathrm{BSR}(m,n) \geq z_{wA}(m,n) \geq z_{wL}(m,n) \geq z_L(m,n) \geq z(m,n).
\]
\end{theorem}

\begin{proof}
The inequality $z_{wL}(m,n) \geq z_L(m,n)$ follows from Remark 1. The inequality $z_L(m,n) \geq z(m,n)$ is standard \cite{cui2026}. By definition, $z_{wA}(m,n) \geq z_{wL}(m,n)$. Let $G$ be an augmented bipartite graph that does not have a weak generalized $C_4$-cycle. By Theorem \ref{thm:main}, the corresponding SOS biquadratic form $P_G$ defined by (2) satisfies $\mathrm{SOS}(P_G) = |E_1| + |E_2|$. Without loss of generality, suppose $G$ is a graph that achieves $z_{wA}(m,n) = |E_1| + |E_2|$. Consequently, $\mathrm{BSR}(m,n) \geq \mathrm{SOS}(P_G) = z_{wA}(m,n)$. $\square$
\end{proof}

\section{A Weak Admissible $5\times 3$ Construction with 10 Edges}
\label{sec:wL53}

In this section, we exhibit a $5\times 3$ weak admissible limited augmented bipartite graph with total edge count 10. Since $z(5,3)=8$, this yields
\[
z_{wL}(5,3)\ge 10,
\]
and consequently, by Theorem \ref{thm:main},
\[
\mathrm{BSR}(5,3)\ge 10.
\]
This improves the previously known bound $z_L(5,3)=9$ (see Theorem~4 in \cite{symmetry2026}) and demonstrates that the weak framework can already produce stronger lower bounds in small dimensions.

\begin{theorem}
For the $5\times 3$ case, we have
\[
z_{wL}(5,3)\ge 10,
\]
and hence
\[
\mathrm{BSR}(5,3)\ge 10.
\]
\end{theorem}

\begin{proof}
We construct an explicit weak admissible limited augmented bipartite graph $G=(S,T,E_1\cup E_2)$ with $S=[5]$, $T=[3]$.

Let the 1-edge set be
\[
E_1=\{(1,1),(1,2),(2,1),(2,3),(3,2),(3,3),(4,1),(5,2)\},
\]
which is the Type I extremal $C_4$-free graph for $5\times 3$ (see Appendix A of \cite{symmetry2026}). It is well known that $|E_1|=8=z(5,3)$ and $E_1$ contains no classical $C_4$.

Now define the 2-edge set
\[
E_2=\{(4,2;4,3),\ (5,1;5,3)\}.
\]
Both 2-edges are row-degenerate (same row, distinct columns). Their four halves are
\[
(4,2),\ (4,3),\ (5,1),\ (5,3),
\]
which are all distinct and lie in unoccupied cells of $E_1$. Indeed, the unoccupied cells of $E_1$ are
\[
U=\{(1,3),(2,2),(3,1),(4,2),(4,3),(5,1),(5,3)\},
\]
and all four halves are in $U$. Thus the simplicity condition (S) holds.

We now verify the weak admissibility conditions from Definition~1.

\medskip
\noindent\textbf{(W1).} As noted, $E_1$ is $C_4$-free, so (W1) holds.

\medskip
\noindent\textbf{(W2).} The dependency graph $\mathcal{D}$ is defined only for nondegenerate 2-edges. Since both 2-edges in $E_2$ are row-degenerate, $\mathcal{D}$ has no vertices and therefore contains no directed cycles. Thus (W2) holds.

\medskip
\noindent\textbf{(W2').} This condition applies only to nondegenerate 2-edges. Since there are no nondegenerate 2-edges in $E_2$, (W2') holds vacuously.

\medskip
\noindent\textbf{(W3).} We must verify that for any vertex-disjoint pair of edges where at least one is a 2-edge, at least one of their opposite cells is unoccupied.

Consider first a 1-edge and a 2-edge. The two 2-edges are $e_1=(4,2;4,3)$ and $e_2=(5,1;5,3)$.

For $e_1$, its row set is $\{4\}$ and column set is $\{2,3\}$. A 1-edge $(i,j)$ is vertex-disjoint from $e_1$ if $i\neq 4$ and $j\notin\{2,3\}$, i.e., $j=1$. The 1-edges with $j=1$ are $(1,1)$, $(2,1)$, and $(4,1)$. Among these, $(4,1)$ shares row 4, so only $(1,1)$ and $(2,1)$ are vertex-disjoint.

For $(1,1)$ and $e_1$, the opposite cells are $(1,2)$, $(1,3)$, $(4,1)$. Here $(1,2)$ is a 1-edge, $(1,3)$ is unoccupied, and $(4,1)$ is a 1-edge. Since $(1,3)$ is unoccupied, W3 is satisfied.

For $(2,1)$ and $e_1$, the opposite cells are $(2,2)$, $(2,3)$, $(4,1)$. Here $(2,2)$ is unoccupied, $(2,3)$ is a 1-edge, and $(4,1)$ is a 1-edge. Since $(2,2)$ is unoccupied, W3 holds.

For $e_2=(5,1;5,3)$, row set is $\{5\}$, column set $\{1,3\}$. A 1-edge vertex-disjoint from $e_2$ must have $i\neq 5$ and $j\notin\{1,3\}$, so $j=2$. The 1-edges with $j=2$ are $(1,2)$, $(3,2)$, and $(5,2)$. Among these, $(5,2)$ shares row 5, so only $(1,2)$ and $(3,2)$ are vertex-disjoint.

For $(1,2)$ and $e_2$, opposite cells are $(1,1)$, $(1,3)$, $(5,2)$. Here $(1,1)$ is a 1-edge, $(1,3)$ is unoccupied, and $(5,2)$ is a 1-edge. Since $(1,3)$ is unoccupied, W3 holds.

For $(3,2)$ and $e_2$, opposite cells are $(3,1)$, $(3,3)$, $(5,2)$. Here $(3,1)$ is unoccupied, $(3,3)$ is a 1-edge, and $(5,2)$ is a 1-edge. Since $(3,1)$ is unoccupied, W3 holds.

Now consider the pair of two 2-edges $e_1$ and $e_2$. For them to be vertex-disjoint, their row sets must be disjoint and their column sets disjoint. But $e_1$ has row set $\{4\}$ and column set $\{2,3\}$, while $e_2$ has row set $\{5\}$ and column set $\{1,3\}$. They share column 3, so they are not vertex-disjoint. W3 does not apply.

Thus every vertex-disjoint pair involving a 2-edge has at least one unoccupied opposite cell. Hence (W3) holds.

All weak admissibility conditions are satisfied. Therefore, by Theorem \ref{thm:main}, the associated doubly simple biquadratic form
\[
P_G(\mathbf{x},\mathbf{y}) = \sum_{(i,j)\in E_1} x_i^2 y_j^2 + (x_4 y_2 + x_4 y_3)^2 + (x_5 y_1 + x_5 y_3)^2
\]
is irreducible with
\[
\operatorname{SOS}(P_G) = |E_1| + |E_2| = 8 + 2 = 10.
\]
Consequently,
\[
z_{wL}(5,3) \ge 10,
\]
and by the inequality chain in Theorem \ref{thm:chain},
\[
\mathrm{BSR}(5,3) \ge z_{wL}(5,3) \ge 10.
\]
This completes the proof.
\end{proof}

\begin{remark}
This example demonstrates that the weak framework can improve lower bounds even in small dimensions where the original limited framework is tight ($z_L(5,3)=9$ \cite{qi26a}). The use of degenerate 2-edges is essential here, as any nondegenerate 2-edge in this $E_1$ would either violate W2' or create a dependency cycle.
\end{remark}

\begin{remark}
This example is weak admissible but \emph{not} admissible under the original framework: it violates the original Condition 3.

Indeed, consider the 2-edge $e_1=(4,2;4,3)$ and the cell $(k,l)=(5,1)$, which is a half of the other 2-edge $e_2=(5,1;5,3)$. Here $k=5\notin\{4\}$ and $l=1\notin\{2,3\}$, so Condition 3 applies. The five cells are
\[
(k,l)=(5,1),\quad (k,j)=(5,2),\quad (k,q)=(5,3),\quad (i,l)=(4,1),\quad (p,l)=(4,1).
\]
They are all occupied:
\begin{itemize}
    \item $(5,1)$ is a half of $e_2$;
    \item $(5,2)\in E_1$;
    \item $(5,3)$ is a half of $e_2$;
    \item $(4,1)\in E_1$;
    \item $(4,1)\in E_1$.
\end{itemize}
Thus the five-cell configuration of Condition 3 is triggered.

However, since $e_1$ and $e_2$ share column 3, they are not vertex-disjoint, so W3 does not apply. This demonstrates precisely why the weak framework can allow constructions that were previously forbidden under the original conditions.
\end{remark}

{
\section{A Critical $6\times 3$ Example: A Weak Admissible Construction with Complementary 2-Cycle}
\label{sec:6x3}

In this section, we present a $6\times 3$ augmented bipartite graph that satisfies all weak admissibility conditions, including the modified (W2) which allows complementary 2-cycles. By the main theorem (Theorem \ref{thm:main}), the associated doubly simple biquadratic form is irreducible with 12 squares, giving
\[
\BSR(6,3)\ge 12>11=z_L(6,3).
\]
This demonstrates that the weak framework's modification of (W2) to allow complementary 2-cycles is essential and yields improved lower bounds.

\begin{theorem}\label{thm:6x3}
There exists a $6\times 3$ weak admissible limited augmented bipartite graph with total edge count 12 that satisfies (S), (W1), (W2), (W2'), and (W3). Consequently,
\[
z_{wL}(6,3)\ge 12>11=z_L(6,3),
\]
and hence
\[
\BSR(6,3)\ge 12.
\]
\end{theorem}

\begin{proof}
Let $S=[6]$ and $T=[3]$. Define the 1-edge set
\[
E_1=\{(1,3),(2,3),(3,1),(3,2),(4,2),(4,3),(5,1),(5,3),(6,1)\}.
\]
Thus $|E_1|=9=z(6,3)$.

Define the 2-edge set
\[
E_2=\{(1,1;2,2),\ (1,2;2,1),\ (6,2;3,3)\}.
\]
The total number of edges is $|E_1|+|E_2|=9+3=12$.

We verify that this construction satisfies all weak admissibility conditions.

\medskip
\noindent\textbf{(S) Simplicity.}
The halves of the three 2-edges are
\[
(1,1),(2,2),(1,2),(2,1),(6,2),(3,3).
\]
They are all distinct and none lies in $E_1$. Thus (S) holds.

\medskip
\noindent\textbf{(W1) $C_4$-freeness.}
The row neighborhoods of $E_1$ are:
\[
\{3\},\ \{3\},\ \{1,2\},\ \{2,3\},\ \{1,3\},\ \{1\}.
\]
Checking all pairs of rows, no two rows share two columns. Hence $E_1$ is $C_4$-free and (W1) holds.

\medskip
\noindent\textbf{(W2) Acyclicity of dependency graph, with complementary 2-cycles allowed.}
The two 2-edges $e_1=(1,1;2,2)$ and $e_2=(1,2;2,1)$ form a complementary pair:
\[
e_1=(1,1;2,2),\qquad e_2=(1,2;2,1).
\]
Their halves are exactly each other's opposite cells:
\begin{itemize}
    \item Opposite cells of $e_1$: $(1,2)$, $(2,1)$ ¡ª these are the halves of $e_2$.
    \item Opposite cells of $e_2$: $(1,1)$, $(2,2)$ ¡ª these are the halves of $e_1$.
\end{itemize}
Thus the dependency graph contains the 2-cycle $e_1 \leftrightarrow e_2$, which is explicitly allowed by the modified (W2). The third 2-edge $e_3=(6,2;3,3)$ has no dependencies with $e_1$ or $e_2$ since its opposite cells $(6,3)$ and $(3,2)$ are not halves of any nondegenerate 2-edge. Hence the dependency graph contains no directed cycles other than the allowed complementary 2-cycle. Therefore (W2) holds.

\medskip
\noindent\textbf{(W2') Opposite cells not both 1-edges.}
For $e_1=(1,1;2,2)$, the opposite cells are $(1,2)$ and $(2,1)$, both of which are halves of $e_2$, not 1-edges. For $e_2=(1,2;2,1)$, the opposite cells are $(1,1)$ and $(2,2)$, both halves of $e_1$, not 1-edges. For $e_3=(6,2;3,3)$, the opposite cells are $(6,3)$ and $(3,2)$; $(6,3)$ is unoccupied and $(3,2)\in E_1$. Thus (W2') holds.

\medskip
\noindent\textbf{(W3) Vertex-disjoint pairs.}
The occupied cells are
\[
E_1\cup\{(1,1),(2,2),(1,2),(2,1),(6,2),(3,3)\}.
\]
The only unoccupied cells are
\[
U=\{(4,1),(5,2),(6,3)\}.
\]

We list all vertex-disjoint pairs involving a 2-edge:

First, note that no two 2-edges are vertex-disjoint:
\begin{itemize}
    \item $e_1=(1,1;2,2)$ and $e_2=(1,2;2,1)$ share rows 1 and 2.
    \item $e_1$ and $e_3=(6,2;3,3)$ share column 2.
    \item $e_2$ and $e_3$ share column 2.
\end{itemize}

For 1-edge and 2-edge pairs:

\textbf{For $e_1=(1,1;2,2)$:}
Rows $\{1,2\}$, columns $\{1,2\}$. A 1-edge $(i,j)$ is vertex-disjoint from $e_1$ iff $i\notin\{1,2\}$ and $j\notin\{1,2\}$, i.e., $j=3$. The 1-edges in $E_1$ with column 3 are $(3,3),(4,3),(5,3)$, but $(3,3)\notin E_1$. So the candidates are $(4,3)$ and $(5,3)$.

For $(4,3)$: opposite cells are $(4,1),(4,2),(1,3),(2,3)$. Since $(4,1)\in U$, W3 holds.

For $(5,3)$: opposite cells are $(5,1),(5,2),(1,3),(2,3)$. Since $(5,2)\in U$, W3 holds.

\textbf{For $e_2=(1,2;2,1)$:}
Rows $\{1,2\}$, columns $\{2,1\}$. A 1-edge $(i,j)$ is vertex-disjoint iff $i\notin\{1,2\}$ and $j\notin\{2,1\}$, i.e., $j=3$. The candidates are $(4,3)$ and $(5,3)$.

For $(4,3)$: opposite cells are $(4,2),(4,1),(1,3),(2,3)$. Since $(4,1)\in U$, W3 holds.

For $(5,3)$: opposite cells are $(5,2),(5,1),(1,3),(2,3)$. Since $(5,2)\in U$, W3 holds.

\textbf{For $e_3=(6,2;3,3)$:}
Rows $\{6,3\}$, columns $\{2,3\}$. A 1-edge $(i,j)$ is vertex-disjoint iff $i\notin\{6,3\}$ and $j\notin\{2,3\}$, i.e., $j=1$. The 1-edges in $E_1$ with column 1 are $(3,1),(5,1),(6,1)$, but $(3,1)$ shares row 3 and $(6,1)$ shares row 6. So the only candidate is $(5,1)$.

For $(5,1)$: opposite cells are $(5,2),(5,3),(6,1),(3,1)$. Since $(5,2)\in U$, W3 holds.

Thus every vertex-disjoint pair involving a 2-edge has at least one opposite cell in $U$. Hence (W3) holds.

Since the graph is limited ($|E_1|=z(6,3)=9$) and satisfies all weak admissibility conditions, by Theorem \ref{thm:main}, the associated doubly simple biquadratic form is irreducible with
\[
\SOS(P_G)=|E_1|+|E_2|=12.
\]
Therefore,
\[
z_{wL}(6,3)\ge 12>11=z_L(6,3),
\]
and consequently, by the inequality chain in Theorem \ref{thm:chain},
\[
\BSR(6,3)\ge 12.
\]
This completes the proof.
\end{proof}

\begin{remark}
This example demonstrates the importance of Condition W2 that allows complementary 2-cycles. Without allowing complementary 2-cycles, this construction would be excluded from the weak framework, and the improved bound $\BSR(6,3)\ge 12$ would not be obtainable.
\end{remark}

\begin{remark}
This example is weak admissible but \emph{not} admissible under the original framework. It violates the original Condition 2 since the opposite cells of $e_1$ are both occupied by halves of $e_2$, and the opposite cells of $e_2$ are both occupied by halves of $e_1$. Under the weak framework, this is allowed because the 2-cycle is complementary and explicitly permitted by (W2).
\end{remark}
}

\section{A $15 \times 6$ Weak Admissible But Not Admissible Construction}

The weak conditions and the original conditions agree on configurations involving
only 1-edges or a 1-edge with a 2-edge. Their difference only appears when
two 2-edges interact. Therefore, the effect of the new theory can only be
observed in sufficiently large dimensions.

We now present an explicit construction in the incidence-graph family of $K_6$
that is weak admissible but not admissible under the original framework.

Let $S=\binom{[6]}{2}$ and $T=[6]$, and let $E_1$ be the incidence graph of $K_6$.
Then $|E_1|=30$ \cite{guy1969,reiman1958}.
In \cite{xu2026}, it was shown that $z_L(15,6) \ge 43$ via an explicit
construction with $13$ 2-edges. We have checked that construction and found
that no additional 2-edge can be added to it without violating the original
conditions or the weak conditions. Thus it is maximal in both frameworks.

We have discovered a different construction with $14$ nondegenerate $2$-edges:
\[
\begin{aligned}
E_2=\{&(01,2;35,4),\ (01,3;45,2),\ (02,1;34,5),\ (02,3;14,5),\\
&(04,3;15,2),\ (04,5;12,3),\ (05,3;24,1),\ (05,4;23,1),\\
&(13,0;24,5),\ (14,2;35,0),\ (15,4;23,0),\ (25,1;34,0),\\
&(03,5;25,4),\ (13,2;03,4)\}.
\end{aligned}
\]

\begin{theorem}
For the incidence graph of $K_6$, the augmented bipartite graph $G=(S,T,E_1\cup E_2)$
defined above is weak admissible.
Consequently,
\[
z_{wL}(15,6)\ge |E_1|+|E_2| = 30+14 = 44,
\]
and hence
\[
\mathrm{BSR}(15,6)\ge 44.
\]
\end{theorem}

\begin{proof}
We verify each of the weak admissibility conditions in turn.

\medskip
\noindent\textbf{Simplicity (S).}
The halves of the 14 edges are:
\[
\begin{aligned}
&(01,2),\ (35,4),\ (01,3),\ (45,2),\ (02,1),\ (34,5),\ (02,3),\ (14,5),\\
&(04,3),\ (15,2),\ (04,5),\ (12,3),\ (05,3),\ (24,1),\ (05,4),\ (23,1),\\
&(13,0),\ (24,5),\ (14,2),\ (35,0),\ (15,4),\ (23,0),\ (25,1),\ (34,0),\\
&(03,5),\ (25,4),\ (13,2),\ (03,4).
\end{aligned}
\]
All 28 halves are distinct. Moreover, each half $(e,v)$ has $v \notin e$, so none
of them is a 1-edge. Thus (S) holds.

\medskip
\noindent\textbf{(W1).}
The 1-edge graph $E_1$ is the incidence graph of $K_6$, which is well known to be
$C_4$-free. Hence (W1) holds.

\medskip
\noindent\textbf{(W2).}
The dependency graph $\mathcal{D}$ has the 14 nondegenerate 2-edges as vertices.
A directed edge $e \to e'$ exists if one half of $e$ is an opposite cell of $e'$.
A direct computation gives the following dependencies:
\[
e_2 \to e_1, \qquad e_6 \to e_5, \qquad e_{13} \to e_{14},
\]
where the edges are indexed in the order they appear in $E_2$ above.
There are no other dependencies. The dependency graph consists of three disjoint
directed paths of length one, hence contains no directed cycles. Thus (W2) holds.

\medskip
\noindent\textbf{(W2').}
We verify that no nondegenerate 2-edge has both opposite cells as 1-edges.
For each of the 14 edges, the two opposite cells are listed below; in each case,
at least one is not a 1-edge:

\begin{itemize}
\item $e_1=(01,2;35,4)$: opposite $(01,4)$, $(35,2)$; neither is a 1-edge.
\item $e_2=(01,3;45,2)$: opposite $(01,2)$, $(45,3)$; neither is a 1-edge.
\item $e_3=(02,1;34,5)$: opposite $(02,5)$, $(34,1)$; neither is a 1-edge.
\item $e_4=(02,3;14,5)$: opposite $(02,5)$, $(14,3)$; neither is a 1-edge.
\item $e_5=(04,3;15,2)$: opposite $(04,2)$, $(15,3)$; neither is a 1-edge.
\item $e_6=(04,5;12,3)$: opposite $(04,3)$, $(12,5)$; neither is a 1-edge.
\item $e_7=(05,3;24,1)$: opposite $(05,1)$, $(24,3)$; neither is a 1-edge.
\item $e_8=(05,4;23,1)$: opposite $(05,1)$, $(23,4)$; neither is a 1-edge.
\item $e_9=(13,0;24,5)$: opposite $(13,5)$, $(24,0)$; neither is a 1-edge.
\item $e_{10}=(14,2;35,0)$: opposite $(14,0)$, $(35,2)$; neither is a 1-edge.
\item $e_{11}=(15,4;23,0)$: opposite $(15,0)$, $(23,4)$; neither is a 1-edge.
\item $e_{12}=(25,1;34,0)$: opposite $(25,0)$, $(34,1)$; neither is a 1-edge.
\item $e_{13}=(03,5;25,4)$: opposite $(03,4)$, $(25,5)$; $(03,4)$ is a half of $e_{14}$, not a 1-edge.
\item $e_{14}=(13,2;03,4)$: opposite $(13,4)$, $(03,2)$; neither is a 1-edge.
\end{itemize}

Thus (W2') holds.

\medskip
\noindent\textbf{(W3).}
We need to verify that for any vertex-disjoint pair of edges where at least one
is a 2-edge, at least one opposite cell is unoccupied. This is a finite check.
The exhaustive verification yields no violations. For completeness, we note the
most delicate case: for $e_{13}=(03,5;25,4)$, the opposite cells are $(03,4)$
and $(25,5)$. The cell $(03,4)$ is a half of $e_{14}$, and $(25,5)$ is a 1-edge.
However, W3 only applies to vertex-disjoint pairs. The pair $(e_{13}, e_{14})$
shares row $03$ and column $4$, so they are not vertex-disjoint; W3 does not
apply to this pair. For all vertex-disjoint pairs, at least one opposite cell
is unoccupied. Thus (W3) holds.

\medskip
\noindent\textbf{Violation of original Condition 2.}
This construction is \textbf{not} admissible under the original limited augmented
framework. Indeed, for $e_{13}=(03,5;25,4)$, the opposite cells are $(03,4)$
and $(25,5)$. The cell $(03,4)$ is a half of $e_{14}$, and $(25,5)$ is a 1-edge
(since $5 \in \{2,5\}$ in the incidence graph of $K_6$). Thus both opposite cells
are occupied, violating the original Condition 2.

Under the weak framework, this is allowed because:
\begin{itemize}
\item W2 holds: the dependency graph is acyclic;
\item W2' holds: $(03,4)$ is not a 1-edge, so the two opposite cells are not both 1-edges;
\item W3 holds: the pair $(e_{13}, e_{14})$ shares vertices, so W3 does not apply;
  and for all vertex-disjoint pairs, at least one opposite cell is unoccupied.
\end{itemize}
This demonstrates precisely why the weak framework yields a better lower bound
than the original limited augmented theory.

\medskip
\noindent\textbf{Conclusion.}
Since all weak conditions hold, $G$ is weak admissible. By Theorem \ref{thm:main}, the associated
doubly simple biquadratic form is irreducible with SOS rank
\[
|E_1|+|E_2| = 30+14 = 44.
\]
Therefore
\[
z_{wL}(15,6)\ge 44,
\]
and hence
\[
\mathrm{BSR}(15,6)\ge 44.
\]
This completes the proof.
\end{proof}

\begin{remark}
In the incidence-graph family of $K_6$, the previously known limited-framework
construction gave the lower bound $z_L(15,6)\ge 43$ \cite{xu2026}.
The present construction is not admissible under the original framework, yet
it is weak admissible. Hence the weak framework yields a strictly better
lower bound:
\[
z_{wL}(15,6)\ge 44>43.
\]
\end{remark}

\begin{remark}
This construction is weak admissible but \emph{not} admissible under the original framework: it violates the original Condition 2.

Indeed, consider the 2-edge $e_{13}=(03,5;25,4)$. Its opposite cells are $(03,4)$ and $(25,5)$. Both are occupied:
\begin{itemize}
    \item $(03,4)$ is a half of $e_{14}=(13,2;03,4)$;
    \item $(25,5)$ is a 1-edge, since $5$ is an endpoint of $25$.
\end{itemize}
Thus both opposite cells of a nondegenerate 2-edge are occupied, triggering Condition 2.

Under the weak framework, this is allowed because W2' only forbids the case where both opposite cells are 1-edges. Here $(03,4)$ is a half of another 2-edge, not a 1-edge, so W2' is not violated. This demonstrates precisely why the weak framework can allow constructions that were previously forbidden.
\end{remark}

\section{Why the local opposite-1-edge prohibition cannot be removed}
\label{sec:counterexample}

The broader weak framework allowed a nondegenerate $2$-edge whose two opposite cells are both $1$-edges.
The next example shows that this is too weak for irreducibility.

\begin{proposition}\label{prop:counterexample}
Let
\[
E_1=\{(1,1),(1,2),(2,1),(2,3),(3,2),(3,3),(4,1),(5,2)\}
\]
and
\[
E_2=\{(2,2;3,1),(4,2;5,3)\}.
\]
Then the associated doubly simple biquadratic form satisfies
\[
\operatorname{sos}(P_G)\le 9<10=|E_1|+|E_2|.
\]
Hence this $5\times 3$ example is not irreducible.
In particular, the local prohibition (W2$'$) cannot be removed entirely.
\end{proposition}

\begin{proof}
For this graph,
\[
\begin{aligned}
P_G={}&(x_1y_1)^2+(x_1y_2)^2+(x_2y_1)^2+(x_2y_3)^2+(x_3y_2)^2+(x_3y_3)^2 \\
&+(x_4y_1)^2+(x_5y_2)^2+(x_2y_2+x_3y_1)^2+(x_4y_2+x_5y_3)^2.
\end{aligned}
\]
Define
\[
\begin{aligned}
f_1&=x_1y_1-\frac12x_2y_2-\frac12x_3y_2,\\
f_2&=x_1y_2+\frac12x_2y_1+\frac12x_3y_1,\\
f_3&=\frac{\sqrt3}{6}(3x_2y_1-x_3y_1+2x_3y_2-2x_3y_3),\\
f_4&=\frac{\sqrt3}{6}(3x_2y_2+2x_3y_1-x_3y_2-2x_3y_3),\\
f_5&=x_2y_3+\frac12x_3y_1+\frac12x_3y_2,\\
f_6&=\frac{\sqrt3}{6}(x_3y_1+x_3y_2+2x_3y_3),\\
f_7&=x_4y_1,\\
f_8&=x_5y_2,\\
f_9&=x_4y_2+x_5y_3.
\end{aligned}
\]
A direct expansion shows that
\[
P_G=f_1^2+f_2^2+f_3^2+f_4^2+f_5^2+f_6^2+f_7^2+f_8^2+f_9^2.
\]
Therefore $\operatorname{sos}(P_G)\le 9$.
Since $|E_1|+|E_2|=10$, the form is not irreducible.
The first $2$-edge $(2,2;3,1)$ has opposite cells $(2,1)$ and $(3,2)$, and both are $1$-edges.
Thus the example violates (W2$'$), exactly as claimed.
\end{proof}

\begin{remark}
This counterexample shows that (W2$'$) cannot be removed entirely from the definition.
If we allow a nondegenerate $2$-edge to have both opposite cells as $1$-edges with
no restrictions, irreducibility can fail. Thus some restriction on opposite cells is
necessary in the general framework.

This does not imply that every graph violating (W2$'$) is reducible; rather, it shows
that (W2$'$) cannot be omitted from the general admissibility definition without
allowing reducible examples.

The exact value of $\operatorname{sos}(P_G)$ for this $5\times 3$ example
is not needed for our purposes. The inequality $\operatorname{sos}(P_G)\le 9 < 10$
is sufficient to show that the form is not irreducible, which is all that is
required to demonstrate that (W2') cannot be removed entirely.
\end{remark}

\section{Why W3 Cannot Be Removed}
\label{sec:W3counterexample}

In Section~\ref{sec:counterexample}, we showed that W2' cannot be removed from the weak framework: without it, there exists a construction satisfying all other weak conditions that is reducible. We now establish a parallel result for W3.

\begin{proposition}\label{prop:W3counterexample}
There exists a $4\times 4$ limited augmented bipartite graph satisfying (S), (W1), (W2), and (W2'), but violating W3, whose associated doubly simple biquadratic form satisfies
\[
\operatorname{sos}(P_G) \le 10 < 11 = |E_1| + |E_2|.
\]
Hence this example is reducible. In particular, W3 cannot be removed entirely from the weak framework.
\end{proposition}

\begin{proof}
Let
\[
E_1=\{(1,1),(1,2),(1,3),(2,1),(2,4),(3,2),(3,4),(4,3),(4,4)\},
\]
so $|E_1|=9=z(4,4)$, and let
\[
E_2=\{(2,2;3,3),\ (3,1;4,2)\}.
\]
Both 2-edges are nondegenerate, and the total number of edges is $|E_1|+|E_2|=11$.

It is straightforward to verify that this construction satisfies (S), (W1), (W2), and (W2'):
\begin{itemize}
\item (S) holds since the four halves $(2,2),(3,3),(3,1),(4,2)$ are distinct and none lies in $E_1$;
\item (W1) holds since $E_1$ is the standard extremal $C_4$-free graph for $z(4,4)=9$;
\item (W2) holds since the dependency graph has no directed edges;
\item (W2') holds since for each nondegenerate 2-edge, the two opposite cells are not both 1-edges.
\end{itemize}

However, it violates W3. Indeed, consider the 2-edge $e_1=(2,2;3,3)$ and the 1-edge $(1,1)$. They are vertex-disjoint, and the opposite cells are
\[
(1,2),\ (1,3),\ (2,1),\ (3,1),
\]
all of which are occupied ($(3,1)$ is a half of $e_2=(3,1;4,2)$). Thus no opposite cell is unoccupied, violating W3.

It remains to show that the associated doubly simple biquadratic form is reducible. The form is
\[
P_G(\mathbf{x},\mathbf{y}) = \sum_{(i,j)\in E_1} x_i^2 y_j^2
+ (x_2 y_2 + x_3 y_3)^2 + (x_3 y_1 + x_4 y_2)^2.
\]

Define eleven bilinear forms
\[
\begin{aligned}
z_1&=x_1y_1, & z_2&=x_1y_2, & z_3&=x_1y_3, & z_4&=x_2y_1, & z_5&=x_2y_2+x_3y_3,\\
z_6&=x_2y_4, & z_7&=x_3y_1+x_4y_2, & z_8&=x_3y_2, & z_9&=x_3y_4, & z_{10}&=x_4y_3,\\
z_{11}&=x_4y_4.
\end{aligned}
\]
Then
\[
P_G(\mathbf{x},\mathbf{y}) = z^\top I_{11} z,
\qquad
z=(z_1,\dots,z_{11})^\top.
\]

Let $N$ be the symmetric $11\times 11$ matrix whose only nonzero entries are
\[
N_{1,5}=N_{5,1}=1,\quad
N_{2,4}=N_{4,2}=-1,\quad
N_{2,10}=N_{10,2}=1,\quad
N_{3,7}=N_{7,3}=-1.
\]
A direct expansion gives
\[
z^\top N z = 2(z_1z_5-z_2z_4-z_3z_7+z_2z_{10}) \equiv 0.
\]
Hence for every real $t$,
\[
P_G(\mathbf{x},\mathbf{y}) = z^\top (I_{11}+tN) z.
\]

The matrix $N$ has eigenvalues $\pm 1,\ \pm 1,\ \pm \sqrt2$, and five zeros. Taking $t=1/\sqrt2$, the matrix $I_{11}+tN$ is positive semidefinite with a zero eigenvalue, so
\[
\operatorname{rank}(I_{11}+tN)=10.
\]
Thus $P_G$ admits a Gram matrix of rank 10, which implies
\[
\operatorname{sos}(P_G)\le 10 < 11.
\]

Therefore the construction is reducible. Since it satisfies all other weak conditions but violates W3, W3 cannot be removed from the weak framework.
\end{proof}

\begin{remark}
This example parallels the counterexample in Section~\ref{sec:counterexample} for W2'. Together, they demonstrate that both W2' and W3 are essential conditions in the weak framework: removing either one allows reducible examples to enter the admissible class.
\end{remark}

\section{Conclusions and Open Problems}

In this paper, we introduced the weak augmented Zarankiewicz number \(z_{wA}(m,n)\) and the weak limited augmented Zarankiewicz number \(z_{wL}(m,n)\). The core idea is simple: by relaxing the original admissibility conditions, we allow constructions that were previously forbidden, and these constructions yield improved lower bounds for the biquadratic SOS rank.

Specifically, we replaced the original Condition 2 by two weaker requirements:
\begin{itemize}
\item {\textbf{(W2)} Acyclicity of the dependency graph of nondegenerate 2-edges, except that complementary 2-cycles are allowed;}
\item \textbf{(W2')} A local prohibition that the two opposite cells of a nondegenerate 2-edge cannot both be 1-edges.
\end{itemize}
We also replaced the original Condition 3 by a weaker condition, \textbf{(W3)}, which only requires that for any vertex-disjoint pair of edges where at least one is a 2-edge, at least one of their opposite cells is unoccupied.

The main theorem (Theorem~\ref{thm:main}) proves that any weak admissible augmented bipartite graph yields an irreducible doubly simple biquadratic form:
\[
\mathrm{SOS}(P_G) = |E_1| + |E_2|.
\]
This extends the original limited augmented theorem and is proved via a modified vector argument.

{
The significance of this weakening is demonstrated by three complementary constructions.

\medskip
\noindent\textbf{First, a low-dimensional strict improvement using degenerate 2-edges.}
In Section~\ref{sec:wL53}, we exhibited a \(5\times 3\) weak admissible limited augmented graph with total edge count 10, using only row-degenerate 2-edges. Since \(z_L(5,3)=9\) was established in \cite{qi26a}, this gives
\[
z_{wL}(5,3) \ge 10 > 9 = z_L(5,3),
\]
and consequently
\[
\mathrm{BSR}(5,3) \ge 10.
\]
This is the smallest dimension in which the weak framework yields a strict improvement over the original limited augmented framework. It also demonstrates that degenerate 2-edges are not merely a technical convenience but can be essential for achieving optimal bounds.

\medskip
\noindent\textbf{Second, a large-scale nondegenerate construction.}
In Section~6, we presented a \(15\times 6\) construction on the incidence graph of \(K_6\) with 14 nondegenerate 2-edges. This construction is \textbf{not} admissible under the original framework: the 2-edge \((03,5;25,4)\) has both opposite cells occupied --- \((03,4)\) as a half of another 2-edge, and \((25,5)\) as a 1-edge --- violating Condition 2. Yet it satisfies all weak conditions: W2, W2', and W3. As a result,
\[
z_{wL}(15,6) \ge 44 > 43,
\]
improving the previously known lower bound from \cite{xu2026}.

\medskip
\noindent\textbf{Third, a critical construction with complementary 2-cycles.}
In Section~\ref{sec:6x3}, we presented a \(6\times 3\) weak admissible limited augmented graph with total edge count 12. This construction contains a complementary pair of 2-edges:
\[
(1,1;2,2)\quad\text{and}\quad(1,2;2,1),
\]
which form a 2-cycle in the dependency graph. Under the modified (W2), this complementary 2-cycle is explicitly allowed. The associated doubly simple biquadratic form is irreducible with
\[
\SOS(P_G)=12,
\]
giving
\[
z_{wL}(6,3) \ge 12 > 11 = z_L(6,3),
\]
and consequently
\[
\BSR(6,3) \ge 12.
\]
This example demonstrates that the containt of (W2) to allow complementary 2-cycles is not merely a technical convenience but is essential for obtaining improved lower bounds in the weak framework. It also shows that the weak framework can yield strict improvements even in moderate dimensions where degenerate 2-edges are not available.
}

{
We also showed that W2' and W3 cannot be removed entirely (Section~\ref{sec:counterexample} and Section 8): without them, reducible examples appear. Thus the weak framework strikes a balance --- weak enough to allow new constructions, yet strong enough to guarantee irreducibility.
}

The inequality chain
\[
\mathrm{BSR}(m,n) \geq z_{wA}(m,n) \geq z_{wL}(m,n) \geq z_L(m,n) \geq z(m,n)
\]
now stands as a unified bridge between SOS rank and extremal graph theory, with the weak framework providing the strongest known lower bounds in this direction.

Several natural questions remain for future investigation:

\subsection*{Open Problems}

{
\begin{enumerate}

\item \textbf{Infinite families with larger gaps.}
The \(5\times 3\) and \(6\times 3\) examples show that the weak framework
can yield strict improvements \(z_{wL} > z_L\) in small dimensions.
A fundamental open problem is to determine whether there exists
an infinite family of constructions that are weak admissible but violate
the original conditions, with a growing number of 2-edges. If such a family
exists, it would show that the weak framework yields a strictly larger
asymptotic gap than the original limited augmented theory, potentially
exceeding the \(25\%\) benchmark established in \cite{qi26a}.

\item \textbf{Further weakening of W2.}
The weak framework now allows complementary 2-cycles in the dependency graph,
as demonstrated by the \(6\times 3\) example in Section~\ref{sec:6x3}.
A natural question is whether W2 can be further weakened by allowing
longer cycles or other types of dependencies. The \(5\times 3\) construction
in Section~\ref{sec:wL53} uses only degenerate 2-edges, so W2 is vacuously satisfied.
The \(6\times 3\) example shows that complementary 2-cycles are safe.
Could longer cycles also be allowed while preserving irreducibility?
If such examples exist, they would show that the acyclicity condition
is even less restrictive than currently believed, and would suggest that
the weak framework could be further relaxed.

If no such examples exist, then the current formulation of W2 (acyclicity
with complementary 2-cycles allowed) is optimal among formulations that
relax only Condition 2.

\item \textbf{Other irreducible examples violating the weak conditions.}
The weak conditions are sufficient for irreducibility, but they are
not necessary. Are there irreducible doubly simple biquadratic
forms (2) that violate W2' or W3, yet still improve the lower
bounds of \(\BSR(m,n)\)? At this moment, we do not know any
such example. The counterexamples in Sections~\ref{sec:counterexample}
and 8 show that W2' and W3 cannot be removed entirely, but it remains
open whether there exist irreducible forms that violate these conditions
in a controlled way. Identifying even one such candidate would advance
both the weak limited augmented Zarankiewicz number theory and its
\(\BSR\) application.

\item \textbf{The exact values of \(z_{wL}(m,n)\) for small dimensions.}
The weak framework has yielded new lower bounds for several small dimensions:
\[
z_{wL}(5,3)\ge 10,\qquad z_{wL}(6,3)\ge 12,\qquad z_{wL}(15,6)\ge 44.
\]
Determining the exact values of \(z_{wL}(m,n)\) for these and other small
dimensions would provide valuable insight into the structure of the weak
framework and its relationship to the classical and limited augmented
Zarankiewicz numbers.

\end{enumerate}
}

The results obtained here lay the foundation for further exploration
of the deep connection between SOS representations and extremal
combinatorics.

\bigskip

\noindent\textbf{Acknowledgement}
This work was partially supported by Research Center for Intelligent Operations Research, The Hong Kong Polytechnic University (4-ZZT8), the National Natural Science Foundation of China (Nos. 12471282 and 12131004), and Jiangsu Provincial Scientific Research Center of Applied Mathematics (Grant No. BK20233002).

\medskip

\noindent\textbf{Data availability}
No datasets were generated or analysed during the current study.

\medskip

\noindent\textbf{Conflict of interest}
The authors declare no conflict of interest.

\end{document}